\documentclass[reqno, 11pt]{amsart}

\usepackage{amsthm,amsmath,amssymb,amsfonts,bm}
\usepackage{stmaryrd}
\usepackage{bbold}
\usepackage{times}
\usepackage[all]{xy}
\usepackage{mathrsfs}
\usepackage{graphicx}
\usepackage{tikz}
\usetikzlibrary{matrix,calc,arrows}
\usepackage{comment}

\theoremstyle{plain}
\newtheorem{thm}{Theorem}
\newtheorem{lem}[thm]{Lemma}
\newtheorem{cor}[thm]{Corollary}
\newtheorem{prop}[thm]{Proposition}

\theoremstyle{definition}

\newtheorem*{claim}{Claim}
\newtheorem{example}[thm]{Example}

\newtheorem*{ack}{Acknowledgment}
\newtheorem*{notat}{Notations and conventions}

\theoremstyle{remark}

\newcommand{\Ker}{\operatorname{Ker}} 
\renewcommand{\Im}{\operatorname{Im}}
\newcommand{\Hom}{\operatorname{Hom}}
\newcommand{\End}{\operatorname{End}}
\newcommand{\Aut}{\operatorname{Aut}}
\newcommand{\Der}{\operatorname{Der}}

\newcommand{\Tens}{\mathsf{T}}

\newcommand{\rhk}{\text{\scriptsize $\rightharpoonup$}}

\renewcommand{\Dot}{\boldsymbol\cdot\,}

\newcommand{\ant}{{\mathsf{S}}}

\newcommand{\acts}{%
  \!\mathrel{\begin{tikzpicture}[scale=.8,baseline=(current  bounding  box.south)] 
  \useasboundingbox (-.6,-.2) rectangle (.1,.2);
  \node at (0,-.06) {} edge[out=210,in=150,loop] ();
\end{tikzpicture}}\!\!\!\!
}

\newcommand{\define}{\ \stackrel{\text{\rm def}}{=}\ }

\newcommand{\Spec}{\operatorname{Spec}}
\newcommand{\Fract}{\operatorname{Fract}}

\newcommand{\GL}{\operatorname{GL}}

\newcommand{\Id}{\operatorname{Id}}

\newcommand{\ch}{\operatorname{char}}

\newcommand{\onto}{\twoheadrightarrow}

\newcommand{\into}{\hookrightarrow}

\newcommand{\gen}[1]{\langle{#1}\rangle}

\newcommand{\til}[1]{\widetilde{#1}}

\newcommand{\Mat}{\operatorname{M}}

\renewcommand{\k}{\mathbb{k}}
\renewcommand{\phi}{\varphi}

\newcommand{\bdot}{\,\text{\raisebox{-.45ex}{$\boldsymbol{\cdot}$}}\,}

\newcommand{\HSpec}{H\text{-}\!\Spec}

\newcommand{\Units}[1]{{#1}^{\boldsymbol \times}}

\newcommand{\cO}{\mathcal{O}}

\newcommand{\cC}{\mathcal{C}}

\renewcommand{\d}{\delta}

\newcommand{\fg}{\mathfrak{g}}

\newcommand{\fp}{\mathfrak{p}}
\newcommand{\fP}{\mathfrak{P}}
\newcommand{\fa}{\mathfrak{a}}

\newcommand{\e}{\varepsilon}

\newcommand{\ZZ}{\mathbb{Z}}

\newcommand{\cat}[1]{\operatorname{\mathsf{#1}}}

\newcommand{\Rep}{\cat{Rep}}

\newcommand{\ModAlg}[1]{{}_{#1}\!\cat{Alg}}

\newcommand{\fin}[1]{{#1}_{\text{\rm fin}}}
\newcommand{\finO}[1]{{#1}_{\text{\rm fin}, \cO}}
\newcommand{\cen}{\mathcal{Z}}

\newcommand{\bn}{\boldsymbol{n}}
\newcommand{\br}{\boldsymbol{r}}
\newcommand{\bs}{\boldsymbol{s}}

\newcommand{\Q}{\operatorname{Q}\!}

\newcommand{\byG}{\text{\sl :} G}
\newcommand{\byH}{\text{\sl :} H}

\newdir{ >}{{}*!/-5pt/\dir{>}}

\newcommand{\iso}{%
\mathrel{\begin{tikzpicture}[scale=.8,baseline=(current  bounding  box.south), bij/.style={above,sloped,inner sep=0.5pt}] 
\useasboundingbox (-.3,-.12) rectangle (.3,.2);
\node(1) at (-.45,0) {};
\node(2) at (.45,0) {};
\path[->]
(1) edge node[bij] {$\sim$}  (2);
\end{tikzpicture}}
}

\begin{document}

\title[Actions of Cocommutative Hopf Algebras]%
{Actions of Cocommutative Hopf Algebras}

\author{Martin Lorenz}
\author{Bach Nguyen}
\author{Ramy Yammine}

\address{Department of Mathematics, Temple University,
    Philadelphia, PA 19122}




\maketitle

\begin{abstract}
Let $H$ be a cocommutative Hopf algebra acting on an algebra $A$. Assuming 
the base field to be algebraically closed and the $H$-action on $A$ to be integral, that is, it is
given by a coaction of some Hopf subalgebra of the finite dual $H^\circ$
that is an integral domain, we stratify the prime spectrum
$\Spec A$ in terms of the prime spectra of certain commutative algebras.
For arbitrary $H$-actions in characteristic $0$, we show that
the largest $H$-stable ideal of $A$ that is
contained in a given semiprime ideal of $A$ is semiprime as well. 
\end{abstract}

\maketitle


\section*{Introduction}

\subsection{} 
\label{SS:action}
Let $H$ be a Hopf algebra over a field $\k$ and let $A$ be an arbitrary associative $\k$-algebra.
An \emph{action} of $H$ on $A$ is given by a $\k$-linear map
$H \otimes A \to A$, $h\otimes a \mapsto h.a$, that makes $A$ into a 
left $H$-module and satisfies the ``measuring'' conditions
$h.(ab) = (h_1.a)(h_2.b)$ and $h.1 = \gen{ \e,h } 1$ for $h\in H$ and $a,b \in A$.
Here, $\otimes = \otimes_\k$\,, $\Delta h = h_1 \otimes h_2$ denotes
the comultiplication of $H$, and $\e$ is the counit. 
We will write $H \acts A$ to indicate such an action.
Algebras equipped with an $H$-action are called left \emph{$H$-module algebras}.
With algebra maps that are also $H$-module maps as morphisms, left $H$-module algebras
form a category, $\ModAlg H$. 
 
For example, an action of a group algebra $\k G$ on $A$ amounts to the datum of a group homomorphism 
$G \to \Aut A$, the automorphism group of the algebra $A$. 
For the enveloping algebra $U\fg$ of a Lie $\k$-algebra $\fg$, an action $U\fg \acts A$ 
is given by a Lie homomorphism $\fg \to \Der A$, the Lie algebra of
all derivations of $A$. In both these prototypical cases, the acting Hopf algebra is cocommutative.
This article investigates the effect of a given action $H\acts A$
of an arbitrary cocommutative Hopf algebra $H$ on the prime and semiprime ideals of $A$,
partially generalizing prior work of the first author on rational actions of algebraic groups
\cite{mL08}, \cite{mL09}, \cite{mL13}. The interesting article \cite{sS10} by Skryabin covers related
territory, but the actual overlap with our work is insubstantial.


\subsection{}
\label{SS:IntroCores}
For now, let $H$ be arbitrary and let $A \in \ModAlg H$.
An ideal $I$ of $A$ that is also an $H$-submodule of $A$ is called an \emph{$H$-ideal}.
The action $H \acts A$ then passes down to an $H$-action on the quotient algebra $A/I$.
The sum of all $H$-ideals that are
contained in an arbitrary ideal $I$, clearly the unique largest $H$-ideal of $A$ 
that is contained in $I$, will be referred to as
the \emph{$H$-core} of $I$ and denoted by $I\byH$. Explicitly,
\begin{equation} 
\label{E:IbyH}
I\byH = \{a \in A \mid H.a \subseteq I\}.
\end{equation}
If $A \neq 0$ and the product of any two
nonzero $H$-ideals of $A$ is again nonzero, then $A$ is said to be \emph{$H$-prime}. 
An $H$-ideal $I$ of $A$ is called \emph{$H$-prime} if $A/I$
is $H$-prime. It is easy to see that $H$-cores of prime ideals are $H$-prime. Denoting
the collection of all $H$-primes of $A$ by $\HSpec A$, we thus obtain a map
$\Spec A \to \HSpec A$\,, $P \mapsto P\byH$.
The fibers 
\begin{equation*} 
\Spec_IA \define \{ P \in \Spec A \mid P\byH = I \}
\end{equation*}
are called the \emph{$H$-strata} of $\Spec A$. 
The stratification $\Spec A = \bigsqcup_{I \in \HSpec A} \Spec_IA$ was pioneered by Goodearl and Letzter \cite{kGeL00} 
in the case of group actions or, equivalently, actions of group algebras.
It has proven to be a useful tool for investigating $\Spec A$, especially
for rational actions of a connected affine algebraic group $G$ over an algebraically closed
field $\k$. In this case, one has a description of each stratum $\Spec_IA$ in terms of the prime spectrum 
of a suitable \emph{commutative} algebra \cite[Theorem 9]{mL09}. 
Our principal goal is to generalize this result to the context of cocommutative Hopf algebras.
This is carried out in Section~\ref{S:Strat} of this article. 
The first two sections serve to deploy
some generalities on actions of Hopf algebras that are needed for the proof, with Section~\ref{S:Cocommutative}
focusing on the cocommutative case.


\subsection{}
\label{SS:IntroIntegralAction}
To state our main result, we make  the following observations; 
see Sections~\ref{S:Actions}-\ref{S:Strat} for details.
Let $H$ be cocommutative and $\k$ algebraically closed. Assume that
the action $H \acts A$ is locally finite, that is, $\dim_\k H.a < \infty$ for all $a \in A$.
Then $A$ becomes a right comodule algebra over the (commutative) finite dual $H^\circ$ of $H$:
\begin{equation}
\label{E:LocFin}
\begin{tikzpicture}[baseline=(current  bounding  box.359), >=latex, scale=.7,
bij/.style={above,sloped,inner sep=0.5pt}]
\matrix (m) [matrix of math nodes, 
column sep=2em, text height=1.5ex, text depth=0.25ex]
{A & A \otimes H^\circ \\ 
a &  a_0 \otimes a_1 \\ };
\draw[->] (m-1-1) edge (m-1-2);
\draw[|->] (m-2-1) edge (m-2-2);
\end{tikzpicture} 
\quad\text{with}\quad h.a = a_{0} \gen{a_1,h}
\hspace{.4in} (h\in H, a\in A).
\end{equation}
The action $H \acts A$ will be called \emph{integral} if the image of the map
\eqref{E:LocFin} is contained in $A \otimes \cO$ for some Hopf subalgebra 
$\cO \subseteq H^\circ$ that is an integral domain. This condition serves as a replacement for connectedness
in the case of algebraic group actions. Assuming it to be satisfied,
it follows that each $I \in \HSpec A$ is in fact a prime ideal of $A$. Consequently, the extended
center $\cC(A/I)= \cen\Q\,(A/I)$ is a $\k$-field, where $\Q\,(A/I)$ denotes the symmetric ring of quotients 
of $A/I$. The action $H \acts A/I$ extends uniquely to an $H$-action on $\Q\,(A/I)$ and this action 
stabilizes the center $\cC(A/I)$.
Furthermore, $\cO \in \ModAlg H$ via the ``hit'' action $\rhk$ that is given by
$\gen{h \rhk f,k} = \gen{f,kh}$ for $f \in \cO$ and $h,k \in H$. 
The actions $\rhk$ and $H \acts \cC(A/I)$ combine to an $H$-action on the tensor product; so
\begin{equation}
\label{E:S_I}
C_I := \cC(A/I) \otimes \cO \in \ModAlg H.
\end{equation}
The algebra $C_I$ is a commutative integral domain.
We let $\Spec^H\!C_I$ denote the subset of $\Spec C_I$ consisting of all prime $H$-ideals of $C_I$\,.

\begin{thm}
\label{T:Strat}
Let $H$ be a cocommutative Hopf algebra over an algebraically closed field $\k$ and 
let $A$ be a $\k$-algebra that is equipped with an integral action $H \acts A$. 
Then, for any  $I \in \HSpec A$, there is a bijection
\[
c \colon \Spec_I A \iso \Spec^H\!C_I 
\]
having the following properties, for any $P,P' \in \Spec_I A$: 
\begin{enumerate}
\item[(i)]
$c(P) \subseteq c(P')$ if and only if $P \subseteq P'$, and
\item[(ii)]
$\Fract(C_I/c(P)) \cong \cC((A/P) \otimes \cO)$.
\end{enumerate}
\end{thm}


\subsection{}
\label{SS:StratExamples}
As a first example, 
let $G$ be an affine algebraic $\k$-group and let $\cO = \cO(G)$ be the algebra of polynomial functions on $G$. 
Then $\cO \subseteq H^\circ$ with $H = \k G$. 
A rational $G$-action on $A$, by definition, is a locally finite action $H \acts A$ such that the image
of \eqref{E:LocFin} is contained in $A \otimes \cO$. If $G$ is connected, then $\cO$ is an integral domain
and so the action is integral. In this setting, Theorem~\ref{T:Strat} is covered by \cite[Theorem 9]{mL09}.
 
Next, let $\fg$ be a Lie $\k$-algebra acting by derivations on $A$  
and assume that every $a \in A$ is contained
in some finite-dimensional $\fg$-stable subspace of $A$. 
With $H = U\fg$, we then have a locally finite action $H \acts A$
and hence a map \eqref{E:LocFin}. If $\ch \k = 0$, then the convolution algebra $H^*$
is a power series algebra over $\k$ and hence
$H^*$ is a commutative domain; see \S\ref{SS:Env} below. Since $H^\circ$ is a subalgebra of $H^*$,
we may take $\cO = H^\circ$ and so we have an integral action. Theorem~\ref{T:Strat} appears to be
new in this case.


\subsection{}
\label{SS:IntroSemiprime}

In Section~\ref{S:Semi}, we show that if $\ch\k = 0$ and $H$ is cocommutative, then the 
core operator $\bdot\byH$ preserves semiprimess. Recall that
an ideal $I$ of $A$ is called \emph{semiprime} if $I$ is
an intersection of prime ideals or, equivalently,
$A/I$ has no nonzero nilpotent ideals.
In Section~\ref{S:Semi}, we turn to the question as to whether 
$I\byH$ is then semiprime as well.
This question may be reformulated in various alternative
ways (Lemma~\ref{L:Reformulation}). In general, the answer is negative: 
even for a cocommutative Hopf algebra $H$, semiprimeness may be lost upon
passage to the $H$-core. For instance, consider the group algebra
$A = \k G$ with its standard $G$-grading. For $G$ finite, this grading amounts to an
action of the Hopf dual $H = (\k G)^*$ on $A$. The only $H$-ideals of $A$  
are $0$ and $A$; so the $H$-core of every proper ideal of $A$ is $0$. If $G$ is abelian, then
$H$ is cocommutative. If the operator $\bdot\byH$ 
is to preserve semiprimeness for any such $G$, then we must require that
$\ch \k = 0$ by Maschke's Theorem. It turns out that this is also sufficient in general:

\begin{thm}
\label{T:Semi}
Let $A \in \ModAlg H$ and assume that 
$H$ is cocommutative and $\ch \k = 0$.
Then $I\byH$ is semiprime for every semiprime ideal $I$ of $A$.
\end{thm}


\subsection{}
\label{SS:Future}

In future work, we hope to pursue the general theme of this article for Hopf algebras
that are not necessarily cocommutative.
In particular, we plan to address ``rationality'' of prime ideals and explore the Dixmier-M{\oe}glin
equivalence in the context of Hopf algebra actions, generalizing the work on group actions
in \cite{mL08}, \cite{mL09}.


\begin{notat}
We work over a base field $\k$ and continue to write $\otimes = \otimes_\k$. 
Throughout, $H$ is a Hopf $\k$-algebra, cocommutative when so specified. The
counit and the antipode of $H$ will be denoted by $\e$ and $\ant$, respectively. Furthermore,
we fix the following notations for the remainder of this article:
\begin{itemize}
\item[]
$A$ will be a left $H$-module algebra, with $H$-action $h \otimes a \mapsto h.a$ 
$(h \in H, a \in A)$;
\item[]
$A^H = \{ a \in A \mid h.a = \gen{\e,h} a \text{ for all } h \in H \}$ is
the subalgebra of $H$-invariants;
\item[]
$B = \Hom_\k(H,A)$ will denote the convolution algebra.
\end{itemize}
\end{notat}


\section{Generalities on actions} 
\label{S:Actions}


\subsection{The convolution algebra $B$}
\label{SS:Conv}

We begin by introducing certain subalgebras  and automorphisms of $B$ as well as some $H$-operations
on $B$ that will be used throughout this article.

First, the counit of $H$ and the action $H \acts A$ give rise to the following maps
$\iota, \d \colon A \to B$, which are easily
seen to be $\k$-algebra embeddings:
\begin{equation*}
\iota a = (h \mapsto \gen{ \e,h } a) \qquad
\d a = (h \mapsto h.a) \hspace{.6in} (h \in H, a \in A).
\end{equation*}
We will generally identify $A$ with $\iota A$ and we will also identify the linear dual $H^*$ 
with the image of the algebra map
$u_* \colon H^* \into B$ that is given by the unit $u \colon \k \to A$. In this way, we view $A$, $H^*$, and
$A \otimes H^*$ as $\k$-subalgebras of $B$, with $A \otimes H^*$
consisting of all $\k$-linear maps $H \to A$ that have finite rank.
Explicitly, $a \otimes f = (h \mapsto a\gen{f,h})$, $\iota a = a \otimes \e$ and $u_* f = 1\otimes f$  
for $a \in A$, $f \in H^*$.

The algebra $B$ becomes a left $H$-module algebra via the
$\rhk$-action, which is defined by
\begin{equation}
\label{E:rhk}
(h\rhk b)(k) = b(kh) \hspace{.6in} (h,k \in H, b \in B).
\end{equation} 
We will write $(B,\rhk)$ when viewing $B \in \ModAlg H$ with \eqref{E:rhk}.
The map $\d \colon A \to (B,\rhk)$ is a morphism in $\ModAlg H$ and $A \otimes H^*$ is an $H$-module subalgebra
of $(B,\rhk)$. 

Our main focus later will be on the following alternative left $H$-module structure on $B$, which
takes into account the given action $H \acts A$:
\begin{equation}
\label{E:dot}
(h\Dot b)(k) = h_1.b(kh_2).
\end{equation} 
If $H$ is cocommutative, then $B \in \ModAlg H$ with  
\eqref{E:dot} as well; see Proposition~\ref{P:DotInv} below.
In general, we may pass between \eqref{E:rhk} and \eqref{E:dot} by means of the $\k$-linear automorphisms 
$\Phi , \Psi \colon B \iso B$ that are defined by
\begin{equation*}
\Phi b = ( h \mapsto  h_1.b(h_2))
\quad \text{and} \quad
\Psi b = ( h \mapsto \ant(h_1).b(h_2)). 
\end{equation*}
These maps are inverse to each other and they 
satisfy the following ``intertwining'' formulas, for any $a \in A$, $b \in B$ and $h \in H$: 
\begin{equation}
\label{E:e*d}
\Phi((\iota a)b) = \d a \, \Phi b, \quad 
\Psi((\d a) b) = \iota a \, \Psi b
\end{equation}
\begin{equation}
\label{E:Inter}
\Phi(h \Dot b) = h \rhk (\Phi b), \quad h \Dot (\Psi b) = \Psi(h \rhk b).
\end{equation}
The $\Psi$-identities follow from those for $\Phi = \Psi^{-1}$.
To verify \eqref{E:e*d}, one checks that both sides send a given $h \in H$ to
$h_1.\big(a\,b(h_2)\big)$. For \eqref{E:Inter}, observe that $(\Phi b)(h) = (h \Dot b)(1)$ and so
$\Phi(h \Dot b)(k) = (kh \Dot b)(1) = (\Phi b)(kh) = (h \rhk (\Phi b))(k)$ for $h,k \in H$ and $b \in B$.
Finally, note that $\Phi$ and $\Psi$ fix the unit element $1_B = u_*(\e)$ and they are right linear 
for the subalgebra $\Hom_\k(H,A^H) \subseteq B$.
In particular, $\Phi$ and $\Psi$ restrict to the identity map on $H^*$ and are right $H^*$-linear.


\subsection{Invariants and the locally finite part of $(B,\rhk)$}
\label{SS:Invar}

The \emph{locally finite part} of any $H$-module algebra $A$ is defined by
\[
\begin{aligned}
\fin A :&= \{ a\in A \mid \dim_\k H.a < \infty \} \\
&= 
\{ a\in A \mid I.a = 0 \text{ for some cofinite ideal $I$ of $H$} \}.
\end{aligned}
\]
Here, ``cofinite'' is short for ``having finite codimension."
The locally finite part $\fin A$ is always
an $H$-module subalgebra of $A$ containing the algebra of invariants, $A^H$.

The following lemma determines the invariants and
the locally finite part of $(B,\rhk)$.

\begin{lem}
\label{L:B_0}
\begin{enumerate}
\item
$(B,\rhk)^H = \iota A \cong A$;
\item
$\fin{(B,\rhk)} = \{ b \in B \mid \Ker b \text{  contains some cofinite ideal of } H\} 
\cong A \otimes H^\circ$.
\end{enumerate}
\end{lem}

\begin{proof}
(a)
Note that $b \in (B,\rhk)^H$ if and only if $b(kh) = \gen{\e,h}b(k)$ for all $h,k \in H$. Equivalently,
$b(h) = \gen{\e,h}b(1)$ for all $h \in H$, which in turn states that $b = \iota(b(1))$.
The assertion about $H$-invariants follows. 

(b)
For any ideal $I$ of $H$ and any $b \in B$, the equality $I \rhk b = 0$ is equivalent to $I \subseteq \Ker b$;
so $b \in \fin{(B,\rhk)}$ if and only if $\Ker b$ contains some cofinite ideal of $H$.
In particular, $\iota A \subseteq \fin{(B,\rhk)}$ and the embedding $u_* \colon H^* \into B$ 
sends the finite dual $H^\circ$ to $\fin{(B,\rhk)}$.
The isomorphism between the subalgebra of $B$ consisting of all finite-rank maps 
and $A\otimes H^*$ (\S\ref{SS:Conv}) restricts to an
isomorphism of subalgebras, $\fin{(B,\rhk)} \cong A \otimes H^\circ$.
\end{proof}

The isomorphisms in Lemma~\ref{L:B_0} will be treated as identifications as in \S\ref{SS:Conv}. In particular,
we will write $A \otimes H^\circ = \fin{(B,\rhk)}$
and identify $A$ and $H^\circ$ with the subalgebras $\iota A = A \otimes \e$ 
and $u_* H^\circ = 1 \otimes H^\circ$, respectively.
As was mentioned, $\fin{(B,\rhk)}$ is an $H$-module subalgebra of $(B,\rhk)$; explicitly, 
with $b = a \otimes f$ $(a \in A, f \in H^\circ)$ formula \eqref{E:rhk} becomes
\begin{equation}
\label{E:rhk'}
h\rhk  (a\otimes f) = a \otimes (h \rhk f) = a \otimes f_1\gen{f_2,h}.
\end{equation}
Furthermore, $\fin{(B,\rhk)} = A \otimes H^\circ $ is also stable under the $H$-operation \eqref{E:dot},
which becomes the standard Hopf operation on tensor products:
\begin{equation}
\label{E:dot'}
h \Dot (a\otimes f) = h_1.a \otimes (h_2 \rhk f) = h_1.a \otimes f_1\gen{f_2,h_2}. 
\end{equation}

More generally, for any Hopf subalgebra $\cO \subseteq H^\circ$, 
we will consider the subalgebra
\[
A_{\cO}:=A \otimes \cO \subseteq \fin{(B,\rhk)}\,.
\] 
Each such $A_{\cO}$ is stable under both
$\rhk$ and $\Dot$ by \eqref{E:rhk'} and \eqref{E:dot'}.

\begin{lem}
\label{L:Stability}
Let $\cO \subseteq H^\circ$ be a Hopf subalgebra.
The $\k$-subspaces of $A_{\cO}$ 
that are stable under right multiplication by $\cO$ and under the
$\rhk$-action \eqref{E:rhk'} are exactly those of the form $W \otimes \cO$,
where $W$ is an arbitrary $\k$-subspace of $A$.
\end{lem}

\begin{proof}
Certainly, each $W \otimes \cO$ is stable under right multiplication by $\cO$
and under \eqref{E:rhk'}.
For the converse, equip $A_{\cO}$ with the ``trivial'' right $\cO$-Hopf module structure: 
the right $\cO$-module and right $\cO$-comodule structures are given by $\Id_A \otimes m_{\cO}$ and
$d = \Id_A\otimes \Delta_{\cO}$, where $m_{\cO}$ and $\Delta_{\cO}$ are the multiplication and comultiplication of
$\cO$, respectively \cite[Examples 10.3, 10.4]{mL18}.
Then $A_{\cO} \into A_{\cO} \otimes \cO \into \Hom_{\k}(H,A_{\cO})$ via $d$,
with $(d b)(h) = h \rhk b$ for $b \in A_{\cO}$.
Now let $V \subseteq A_{\cO}$ be a $\k$-subspace that is
stable under right multiplication by $\cO$ and under \eqref{E:rhk'}. Then $d V \subseteq 
\Hom_{\k}(H,V) \cap (A_{\cO} \otimes \cO) = V \otimes \cO$ and so $V$ is a  
$\cO$-Hopf submodule of $A_{\cO}$. By the Structure Theorem for Hopf modules
\cite[\S 10.1.2]{mL18}, $V$ is generated as right $\cO$-module by the 
subspace of $\cO$-coinvariants, $V^{\text{co}\,\cO} = \{ v \in V \mid \d v = v \otimes 1\}$.
Since $V^{\text{co}\,\cO} \subseteq (A_{\cO})^{\text{co}\,\cO} = A \otimes \cO^{\text{co}\,\cO} = A \otimes\k$,
it follows that $V = W \otimes \cO$ with $W = V^{\text{co}\,\cO}$. 
\end{proof}


\subsection{Coefficient Hopf algebras of locally finite actions}
\label{SS:Coeff}

\subsubsection{}
\label{SSS:Coeff1}
Assume that the action $H \acts A$ is locally finite, that is, $A = \fin A$. 
Then, since $\d \colon A \to (B,\rhk)$ is
a morphism in $\ModAlg H$ (\S\ref{SS:Conv}), the image $\d A$ is contained in 
$\fin{(B,\rhk)} = A \otimes H^\circ$ and
$A$ becomes a right $H^\circ$-comodule algebra via $\d$ \cite[Proposition 10.26]{mL18}.
Any Hopf subalgebra $\cO \subseteq H^\circ$ such that 
$\d A$ is contained in the subalgebra $A_{\cO} = A \otimes \cO \subseteq \fin{(B,\rhk)}$ will be called a 
\emph{coefficient Hopf algebra} for $H \acts A$. 
We then have the following version of \eqref{E:LocFin}, which makes $A$ a right $\cO$-comodule algebra:
\begin{equation}
\label{E:CoModAlgO}
\begin{tikzpicture}[baseline=(current  bounding  box.358),  >=latex, scale=.7]
\matrix (m) [matrix of math nodes, 
column sep=2em, text height=1.2ex, text depth=0.25ex]
{\d \colon A & A_{\cO} \\ \hspace{.2in} a & a_0 \otimes a_1 \\};
\draw[->] (m-1-1) edge (m-1-2);
\draw[|->] (m-2-1) edge (m-2-2);
\end{tikzpicture} 
\quad\text{with}\quad h.a = a_{0} \gen{a_1,h}
\hspace{.3in} (h\in H, a\in A).
\end{equation}
Any coefficient Hopf algebra for  $H \acts A$ will also
serve as a coefficient Hopf algebra for the $H$-action on quotients of $A$ modulo $H$-ideals
and on $H$-module subalgebras of $A$, and the intersection of all coefficient Hopf subalgebras for $H \acts A$
is the unique smallest one. 

\begin{lem} 
\label{L:Coeff'}
Let $H \acts A$ be locally finite and let $\cO \subseteq H^\circ$ be a coefficient
Hopf algebra. Then:
\begin{enumerate}
\item
All subspaces of $A_{\cO}$ of the form
$W \otimes \cO$, where $W \subseteq A$ is an $H$-stable $\k$-subspace, are stable under the maps 
$\Phi, \Psi$ (\S{\rm\ref{SS:Conv}}).
\item
If $I$ is an ideal of $A$, then
$I\byH = \Psi(I \otimes \cO) \cap A$.
\end{enumerate}
\end{lem}

\begin{proof}
(a)
If $W \subseteq A$ is $H$-stable, then $\d W \subseteq W \otimes \cO$. Thus,
for any $a \in W$ and $f \in \cO$, 
right $\cO$-linearity of $\Phi$ and  \eqref{E:e*d} now
give $\Phi(a\otimes f) =\d(a)f = a_0 \otimes a_1f \in W \otimes \cO$. Similarly, $\Psi(a \otimes f) = 
a_0 \otimes \ant^*(a_1)f \in W \otimes \cO$\,,
proving stability of $W \otimes \cO$ under $\Phi$ and $\Psi$. 

(b)
By \eqref{E:IbyH}, $I\byH = \d^{-1}(\Hom_\k(H,I)) = \d^{-1}(I  \otimes \cO)$. 
Putting $I' = \Psi(I \otimes \cO)$ and using the identity $\Phi \circ \iota = \d$ from \eqref{E:e*d},
we obtain $I\byH =  \d^{-1}(I  \otimes \cO) = (\iota)^{-1}(I') = I' \cap A$.
\end{proof}

\subsubsection{}
\label{SSS:Coeff2}
The action $H \acts A$ is certainly locally finite if the $H$-module algebra
$A$ under consideration is generated as $\k$-algebra by a finite-dimensional
$H$-stable subspace $V \subseteq A$. Assume this to be the case and
let $\rho \colon H \to \End_\k(V)$ denote 
the algebra map given by the operation of $H$ on $V$. 
Fixing a $\k$-basis $(v_i)_1^n$ of $V$ and 
letting  $(v_i^*)$ denote the dual basis of $V^*$, we obtain the basis $(v_i \otimes v_j^*)$ of 
$\End_\k(V) \cong V \otimes V^*$ and linear forms $\rho_{i,j}$ such that
\begin{equation}
\label{E:CoeffV}
\rho h = \sum_{i,j} v_i \otimes v_j^* \gen{\rho_{i,j},h} \qquad (h \in H).
\end{equation}
The isomorphism $\End_\k(V) \cong \Mat_n(\k)$ given by our choice of basis for $V$
allows us to write $\rho = (\rho_{i,j}) \colon H \to \Mat_n(\k)$: the scalar $\gen{\rho_{i,j},h} \in \k$
is the $(i,j)$-entry of the matrix $\rho h$. Thus,  $\rho_{i,j} \in H^\circ$, because 
$\rho_{i,j}$ vanishes on the cofinite ideal $\Ker\rho$, and
$\Delta\rho_{i,j} = \sum_k \rho_{i,k} \otimes \rho_{k,j}$.
On the generating subspace $V \subseteq A$, the algebra map \eqref{E:CoModAlgO} takes the form
\begin{equation}
\label{E:CoModAlgO'}
\d v = \sum_{i,j} v_i\gen{v_j^*,v} \otimes \rho_{i,j} \qquad (v \in V)
\end{equation}
and we may take $\cO$ to be the Hopf subalgebra of $H^\circ$ that is generated by the 
matrix coefficient functions $\rho_{i,j}$. 
If $H$ is involutory, then $\cO$ is the $\k$-subalgebra of $H^\circ$ that is generated
by all $\rho_{i,j}$ and $\ant^*\rho_{i,j}$. 
We will call $\cO$ the \emph{coefficient Hopf algebra} of the 
representation $V \in \Rep H$; a more general situation is discussed in \cite[Exercise 9.2.3]{mL18}.


\section{The cocommutative case} 
\label{S:Cocommutative}


\subsection{The $\Dot$-action}
\label{SS:Act2}

\subsubsection{}
\label{SSS:centralizing}

We recall some general ring-theoretic material that will be tacitly used 
in the next proposition and throughout the remainder of this article.
A ring homomorphism $f \colon R \to S$ is called \emph{centralizing} if the ring $S$ is generated
by the image $fR$ and its centralizer in $S$. If $f$ makes $S$ a free $R$-module having a basis 
that centralizes $fR$, then $f$ is called \emph{free centralizing}. 
Any centralizing homomorphism $f$ restricts to a map of centers, $\cen R \to \cen S$; 
the assignment $I \mapsto (fI)S$ sends (two-sided) ideals of $R$ to ideals of $S$; and 
$P \mapsto f^{-1}P$ gives a well-defined map $\Spec S \to \Spec R$ \cite[\S1.5]{mL08}.

\subsubsection{}

The proposition below shows that, for $H$ cocommutative,
we may view $B$ and various subalgebras of $B$ as $H$-module algebras with
action \eqref{E:dot}, which we will refer to as the $\Dot$-action, rather than the 
$\rhk$-action \eqref{E:rhk}. When using the latter $H$-action, we
will continue write $(B,\rhk)$; otherwise, the $\Dot$-action is assumed. 

\begin{prop}
\label{P:DotInv}
Let $H$ be cocommutative and let $\cO \subseteq H^\circ$ be a coefficient Hopf algebra
for the action $H \acts \fin A$\,. Then:
\begin{enumerate}
\item
$\Phi, \Psi$ are algebra automorphisms of $B$ stabilizing the subalgebra $\finO A:= \fin A \otimes \cO$.
\item
$B \in \ModAlg H$ with the $\Dot$-action and $B^H = \Psi A$.
\item
$A_{\cO}$ is an $H$-module subalgebra of $B$, with $\fin{(A_{\cO})} = \finO A$ and
$(A_{\cO})^H = \Psi(\fin A)$. The inclusion $(A_{\cO})^H \subseteq \fin{(A_{\cO})}$
is free centralizing.
\item
The following maps are bijections that are inverse to each other: 
\[
\begin{tikzpicture}[baseline=(current  bounding  box.359), >=latex, scale=.7,
bij/.style={above,sloped,inner sep=0.5pt}]
\matrix (m) [matrix of math nodes, 
column sep=2em, text height=1.5ex, text depth=0.25ex]
{\{ \text{\rm ideals of } \fin A \} & \{ \text{\rm $H$-ideals of } \finO A\} & \{ \text{\rm ideals of } (A_{\cO})^H \}  \\[.5em] 
I & I':= \Psi(I \otimes \cO) & I' \cap (A_{\cO})^H \\[-.2em]
\Phi(J'') \cap \fin A &  J'':=J\finO A &  J  \\  };
\draw[<->] (m-1-1) edge node[bij] {$\sim$} (m-1-2);
\draw[<->] (m-1-2) edge node[bij] {$\sim$} (m-1-3);
\draw[|->] (m-3-3) edge (m-3-2);
\draw[|->] (m-3-2) edge (m-3-1);
\draw[|->] (m-2-1) edge (m-2-2);
\draw[|->] (m-2-2) edge (m-2-3);
\end{tikzpicture}
\]
\end{enumerate}
\end{prop}

\begin{proof}
(a)
For $b,b' \in B$ and $h \in H$, one computes using cocommutativity,
\[
\begin{aligned}
\Phi(bb')(h) &= h_1.(b(h_2)b'(h_3)) = (h_1.b(h_3))(h_2.b'(h_4)) \\
&= (h_1.b(h_2))(h_3.b'(h_4)) = (\Phi(b)\Phi(b'))(h).
\end{aligned}
\]
Thus $\Phi$ is multiplicative and hence it is
an algebra automorphisms of $B$; likewise for $\Psi = \Phi^{-1}$. Stability of 
the subalgebra $\finO A$ under $\Phi$ and $\Psi$
follows from Lemma~\ref{L:Coeff'}(a) for $\fin A$.

(b)
Since $\Psi$ is an algebra automorphism of $B$,
the intertwining formula \eqref{E:Inter} together with the fact that 
$(B,\rhk) \in \ModAlg H$ implies that 
$B = (B, \Dot) \in \ModAlg H$ as well. Also by \eqref{E:Inter}, $\Psi$ yields
an algebra isomorphism $(B,\rhk)^H \cong B^H$ and so
Lemma~\ref{L:B_0}(a) gives the isomorphism
$\Psi\circ\iota \colon A \iso (B,\rhk)^H \iso B^H$. Identifying $A$ with $\iota A = A \otimes \e$ as usual, 
we obtain the claimed equality $B^H = \Psi A$. 

(c)
The subalgebra $A_\cO \subseteq B$ is stable under the $\Dot$-action by \eqref{E:dot'} and
so it is an $H$-module subalgebra of $B$. The equality $\fin{(A_{\cO})} = \finO A$ 
follows from \cite[Corollary 6]{sKmLbNrY19}, because the $H$-action $\rhk$ on $\cO$ is
locally finite. Since $\ant$ is an involution of $H$, the formula 
$(\Psi a)(h) = \ant(h_1).a\gen{\e,h_2} =  \ant(h).a$ $(a \in A, h \in H)$
shows that $\Psi a$ 
vanishes on some cofinite ideal of $H$ if and only if $a \in \fin A$. In that case, 
$\Psi a \in \finO A \subseteq A_{\cO}$ by (a). Thus, in view of part (b) and Lemma~\ref{L:B_0}(b),
$\Psi(\fin A) =  B^H \cap \fin{(B,\rhk)} = (A_{\cO})^H$ as asserted.
Finally, since $\fin{(A_{\cO})} = \finO A = \Psi(\finO A)$ is generated by the commuting subalgebras 
$\Psi(\fin A) = (A_{\cO})^H$ and $\Psi\cO$, with $\Psi\cO$ providing
an $(A_{\cO})^H$-basis of $\fin{(A_{\cO})}$, the inclusion 
$(A_{\cO})^H \subseteq \fin{(A_{\cO})}$ is free centralizing.

(d)
Let $J$ be an ideal of $(A_{\cO})^H$. Since the inclusion 
$(A_{\cO})^H \subseteq \fin{(A_{\cO})} = \finO A$ is centralizing by (c), $J \finO A$ is an
ideal of $\finO A$, clearly an $H$-ideal. Further,
$J \finO A \cap (A_{\cO})^H = J$ by freeness of $\finO A$ over $(A_{\cO})^H$.
Now let $L$ be any $H$-ideal of $\finO A$. Then $\Phi L$ is an ideal of $\finO A$ that is
stable for the $\rhk$-action by \eqref{E:Inter}. Thus, Lemma~\ref{L:Stability} implies that
$\Phi L = I \otimes \cO$ for some ideal $I$ of $\fin A$.
Therefore, $L = \Psi(I \otimes \cO) = \Psi I\,\Psi \cO = (\Psi I) \finO A$
with $\Psi I$ being an ideal of  $(A_{\cO})^H = \Psi(\fin A)$. This proves that extension and contraction
give inverse bijections between the sets
of ideals of $(A_{\cO})^H$ and $H$-ideals of $\finO A$. The bijection between the sets of ideals of
$(A_{\cO})^H$ and of $\fin A$ is a consequence of the equality $\Psi(\fin A) =  (A_{\cO})^H$.
\end{proof}


\subsection{Integral actions}
\label{SS:IntAct}

\subsubsection{}
\label{SSS:IntAct1}

A locally finite action $H \acts A$ will be called \emph{integral} if it has a coefficient Hopf algebra that is
a commutative integral domain. Even though $H$ need not a priori be cocommutative,
any integral $H$-action factors through a cocommutative Hopf quotient of $H$.

\begin{prop}
\label{P:IntPrime}
Let $H \acts A$ be integral and $\k$ algebraically closed. Then $P\byH$ is prime
for every $P \in \Spec A$. Furthermore, $\HSpec A$ is the set of all prime $H$-ideals of $A$.
\end{prop}

\begin{proof}
Let $\cO \subseteq H^\circ$ be a coefficient Hopf algebra of the action $H \acts A$ that is an integral domain,
and assume that $H$ is cocommutative, as we may. Then $\Psi$ 
is an algebra automorphism of $B$ that restricts to
an automorphism of the subalgebra $A_{\cO}$ (Proposition~\ref{P:DotInv}). For any 
$P \in \Spec A$, the ideal $P \otimes \cO$ of $A_{\cO}$ is prime \cite[Lemma 11.19]{mL18},
and hence 
\begin{equation}
\label{E:IntPrime}
P' := \Psi(P \otimes \cO) \in \Spec A_{\cO}.
\end{equation}
By Lemma~\ref{L:Coeff'}(b), $P\byH = P' \cap A$, which is a prime ideal of $A$,
because $A \into A_{\cO}$
is centralizing. The final assertion also follows, because
the map $\Spec A \to \HSpec A$\,, $P \mapsto P\byH$, is surjective for any locally
finite action $H \acts A$ \cite[Exercise 10.4.4]{mL18}.
\end{proof}

\subsubsection{}
\label{SSS:IntAct2}
Returning to the situation considered in \S\ref{SSS:Coeff2}, assume that 
$A$ is affine, generated by an
$H$-stable subspace $V \subseteq A$ with $n = \dim_\k V< \infty$. We use our earlier notation, but we 
now also assume that $H$ is cocommutative and so involutory.
The $\k$-subalgebra $\cO \subseteq H^\circ$ that is generated
by the functions $\rho_{i,j}$ and $\ant^*\rho_{i,j}$ is thus a coefficient Hopf algebra for
$H \acts A$.

\begin{example}[Group algebras]
\label{EX:IntActkG}
Let $H = \k G$ be a group algebra. Then $H^*$ is the algebra $\k^G$ of all functions $G \to \k$ with pointwise
addition and multiplication. The subalgebra $H^\circ$ is commonly 
referred to as the algebra of representative functions on $G$ 
and denoted by $R_\k(G)$ (e.g., \cite[Chapter 1]{gH81}). 
For any $g \in G$, the determinant $\det \rho g$ is nonzero
and $\gen{\ant^*\rho_{i,j},g} = \gen{\rho_{i,j},g^{-1}} = \tfrac{1}{\det \rho g} \gen{c_{j,i},g}$,
where $\gen{c_{j,i},g}$ denotes the $(j,i)$-cofactor of the matrix $\rho g \in \GL_n(\k)$, a polynomial in the
entries of $\rho g$. Thus, $c_{j,i} \in R:=\k[\rho_{i,j} \mid i,j = 1,\dots,n]$ and
$\ant^*\rho_{i,j} = \tfrac{c_{j,i}}{\det \rho} \in R[(\det \rho)^{-1}]$, 
the subalgebra of $H^\circ$ that is generated by the 
functions $\rho_{i,j}$ and $(\det\rho)^{-1} = \ant^*\det\rho$. We obtain 
\[
\cO = R[(\det \rho)^{-1}] = \k[(\det \rho)^{-1}, \rho_{i,j} \mid i,j = 1,\dots,n].
\] 
Now let $\k$ be algebraically closed.
The group $\GL_n(\k)$ is affine algebraic, with associated Hopf algebra $\cO(\GL_n)$ as in \cite[Example 9.19]{mL18}. 
The Hopf algebra $\cO$ is the image of $\cO(\GL_n)$ under restriction 
to $\rho G \le \GL_n(\k)$. Finally, $\cO$ is a domain if and only if $\rho G$ is an irreducible subset
of $\GL_n(\k)$ in the Zariski topology or, equivalently, the closure $\overline{\rho G}$ is a connected affine
algebraic group.
\end{example}

\begin{example}[Enveloping algebras]
\label{EX:IntActUg}
Let $\fg$ be a Lie $\k$-algebra and $H = U\fg$ its enveloping algebra.
If $\ch\k = 0$, then $H^*$ is a (commutative) domain; 
see \cite[Chap.~II, \S 1, n$^{\rm o}$ 5]{nB72} or Section~\ref{SS:Env} below.
Therefore, the subalgebra $H^\circ$ and all coefficient Hopf algebras $\cO$ are 
integral domains as well in characteristic $0$.
The situation is different for $\ch\k = p > 0$.
Indeed, in this case, $\gen{f^p,\fg} = 0$ for any $f \in H^*$. If $f \in H^\circ$ is grouplike 
or, equivalently, an algebra map, then so is $f^p$ 
and hence $\Ker f^p$ is the ideal of $H$ that is generated by $\fg$.
Therefore, $f^p = \e$ and so $(f - \e)^p = 0$. If $\fg \neq [\fg,\fg]$ then we may choose $f \neq \e$, giving
a  nonzero nilpotent element of $H^\circ$.
\end{example}


\subsection{The symmetric ring of quotients and the extended center}
\label{SS:C(R)}

\subsubsection{}

We briefly recall some background material
on the symmetric ring of quotients $\Q R$ of an arbitrary ring $R$ and 
its center, $\cC R:= \cen(\Q R)$, the so-called \emph{extended center} of $R$. See
\cite[Appendix E]{mL18} for details. The ring $R$ is a subring of $\Q R$ and $\cC R$ 
coincides with the centralizer of $R$ in $\Q R$. In particular, $\cen R \subseteq \cC R$.
If $\cen R = \cC R$, then $R$ is called \emph{centrally closed}. In general, we may consider 
the following subring of $\Q R$, possibly strictly larger than $R$:
\[
\til R:= R(\cC R) \subseteq \Q R.
\]
If $R$ is semiprime, then $\til R$ is a centrally closed ring
\cite[Theorem 3.2]{wBwM79}, called the \emph{central closure} of $R$.
If $R$ is a $\k$-algebra, then so is $\til{R}$, because $\cen R \subseteq \cC R = \cen \til{R}$.

\subsubsection{}
\label{SSS:QA}

The next lemma concerns the extension of a given action $H\acts A$ to an action on $\Q A$. 
Part (a), which summarizes known facts, shows that this is
always possible, in a unique way, in the situation that we are interested in. Indeed, any cocommutative 
Hopf algebra over an algebraically closed field is pointed \cite[Lemma 8.0.1]{mS69}. However, local
finiteness of an action may be lost in the process. For instance, consider the action of the
group $\Units \k$ on the polynomial algebra $A = \k[x]$ that is determined by $\lambda.x = \lambda x$
for $\lambda \in \Units\k$ and the extended action on 
the field $\Q A = \Fract A = k(x)$. If $\k$ is infinite, then $\fin{\k(x)} = \k[x^{\pm 1}]$. Part (b) of the lemma focuses on
the action $H \acts \fin{(\Q A)}$ rather than $H \acts \Q A$.

\begin{lem}
\label{L:C(R)}
Let $H$ be pointed cocommutative. Then:
\begin{enumerate}
\item
The $H$-action on $A$ extends uniquely to an action $H \acts \Q A$ and this action 
stabilizes the subalgebras $\til A = A(\cC A)$ and $\cC A$.
\item
If the action $H \acts A$ is locally finite with coefficient Hopf algebra $\cO$, then the extended action 
$H \acts \fin{(\Q A)}$ also has coefficient Hopf algebra $\cO$.
\end{enumerate}
\end{lem}

\begin{proof}
(a)
Since $H$ is pointed, \cite[Cor. 3.5]{sM93a} tells us 
that the $H$-action on $A$ extends uniquely to an action on $\Q A$. 
This action stabilizes the center, $\cC A$, because $H$ is cocommutative \cite[Prop. 4]{mC86}.
Therefore, $\til A$ is stable as well.

(b)
Put $R = \fin{(\Q A)}$. So $A \subseteq R$ and the action $H \acts \Q A$ restricts to a locally finite
action $H \acts R$ extending the $H$-action on $A$. 
We must show that the map $\d \colon R \to R \otimes H^\circ$
(\S\ref{SS:Coeff}) has image in $R_{\cO} = R \otimes \cO$, given that $\d A \subseteq A_{\cO}$.
Fix $r \in R$ and consider the subspace $V = H.r \subseteq R$, which is 
finite dimensional and $H$-stable. Therefore, 
the ideal $I = \{ a\in A \mid aAV \subseteq A \}$ of $A$ has zero (left and) right annihilator in $R$ 
\cite[Proposition E.1]{mL18} and $I$ is an $H$-ideal as is readily verified using the
following standard identity for $H$-module algebras: 
\begin{equation}
\label{E:HAlgId}
(h.x)y = h_1.(x(\ant(h_2).y)) \qquad (h \in H, \, x,y \in R).
\end{equation}
By Lemma~\ref{L:Coeff'}(a), $\Phi(I \otimes \cO) = I\otimes \cO$.
Furthermore, by Proposition~\ref{P:DotInv}(a), $\Phi$ gives an algebra automorphism of 
$R \otimes H^\circ$ stabilizing $A_{\cO}$ and satisfying $\Phi \circ \iota = \d$ by \eqref{E:e*d}.
Therefore,
\[
(I \otimes \e) \d V \subseteq \Phi((I \otimes \cO)(V \otimes \e)) = \Phi(IV \otimes \cO) 
\subseteq \Phi(A\otimes \cO) = A_{\cO}\,.
\]
Thus, $(I \otimes \e)\d r \subseteq R_{\cO}$\,.
Since $\d r \in R \otimes H^\circ$ and $I$ has zero right annihilator in $R$, it follows that 
$\d r \in R_{\cO}$, as desired.
\end{proof}


\section{Prime strata} 
\label{S:Strat}


\subsection{Prime correspondences}
\label{SS:Corr}

The central closure $\til R$ of a prime ring $R$ is also prime and its center,
the extended center $\cC R$, is a field. Thus, for an arbitrary ring $R$, we may 
associate to any $P \in \Spec R$ the field $\cC(R/P)$, 
called the \emph{heart} of $P$. 

\begin{prop} 
\label{P:bij1}
Let $R$ be a centrally closed prime $\k$-algebra and put $K = \cC R$. Let $S$ be any
$K$-algebra and put $U = R \otimes_K S$. Then we have the following bijections, which are inverse to each other:
\[
\begin{tikzpicture}[baseline=(current  bounding  box.359), >=latex, scale=.7,
bij/.style={above,sloped,inner sep=0.5pt}]
\matrix (m) [matrix of math nodes, 
column sep=2em, text height=1.5ex, text depth=0.25ex]
{\{ P \in \Spec U \mid P \cap R = 0\} & \Spec S \\[.5em] 
P & P\cap S \\[-.2em] U\fp = R\otimes_K\fp &  \fp \\ };
\draw[<->] (m-1-1) edge node[bij] {$\sim$} (m-1-2);
\draw[|->] (m-2-1) edge (m-2-2);
\draw[|->] (m-3-2) edge (m-3-1);
\end{tikzpicture}
\]
This correspondence preserves hearts: $\cC(U/P) \cong \cC(S/P \cap S)$.
Moreover, if $U \in \ModAlg H$  with $S$ being an $H$-module subalgebra,
then $H$-ideals are matched with $H$-ideals.
\end{prop}

\begin{proof}
Apart from the last assertion, involving an $H$-action, the proposition is identical with
\cite[Proposition 5]{mL09}. The bijections in the proposition,
contraction and extension of ideals, both evidently send $H$-ideals to $H$-ideals in the given situation.
\end{proof}

\begin{prop} 
\label{P:bij2}
Assume that $H$ is pointed cocommutative and let $R,T \in \ModAlg H$, with
$R$ being prime. Put $V:= R \otimes T \subseteq \til V:= \til R \otimes T$, where
$\til R$ denotes the central closure of $R$. Then there is a bijection
\[
\begin{tikzpicture}[baseline=(current  bounding  box.359), >=latex, scale=.7,
bij/.style={above,sloped,inner sep=0.5pt}]
\matrix (m) [matrix of math nodes, 
column sep=2em, text height=1.5ex, text depth=0.25ex]
{\{ P \in \Spec \til{V}  \mid P \cap \til{R} = 0 \} & 
\{ Q \in \Spec V \mid Q \cap R  = 0 \} \\[.5em] 
P & P \cap V \\ };
\draw[->] (m-1-1) edge node[bij] {$\sim$} (m-1-2);
\draw[|->] (m-2-1) edge (m-2-2);
\end{tikzpicture}
\]
This bijection is an order isomorphism for $\subseteq$ and it preserves hearts.
Moreover, viewing $\til R \in \ModAlg H$ with
the extended $H$-action (Lemma~\ref{L:C(R)}) and $V, \til V \in \ModAlg H$ with 
the standard $H$-action on tensor products, the bijection gives a bijection on $H$-stable primes.
\end{prop}

\begin{proof}
Again, this proposition is covered by \cite[Proposition 6]{mL09} except for the statement about $H$-stability.
To justify this assertion, note that tensor products of $H$-module algebras are again $H$-module algebras,
with the standard $H$-action via $\Delta$, because $H$ is cocommutative.
Thus, $\til V \in \ModAlg H$ and all of $R$, $\til R$, $T$ and $V$ are $H$-module subalgebras of $\til V$.
If $P$ is an $H$-ideal of $\til V$, then its contraction, $P \cap V$, is clearly an $H$-ideal of $V$.
Conversely, let $Q \in \Spec V$ with $Q \cap R = 0$ be $H$-stable. In order to show that
the preimage of $Q$ under the bijection in the
proposition is an $H$-ideal of $\til V$, we recall the construction of
the preimage from \cite{mL09}. 
The canonical epimorphism
$\pi \colon V \onto W:= V/Q$ is a map in $\ModAlg H$, and the restriction 
$\rho:= \pi\big|_R \colon R \into W$ is an embedding of prime algebras and a centralizing map in $\ModAlg H$.
By  \cite[Lemma 4]{mL09}, $\rho$ extends uniquely to an embedding of central closures,
$\til\rho \colon \til R \into \til{W}$, which is also centralizing.

\begin{claim}
$\til\rho$ is a map in $\ModAlg H$ for the extended action $H \acts \til W$ (Lemma~\ref{L:C(R)}).
\end{claim}
To prove the claim, let
$\til r \in \til R$ and $h \in H$. Fix a nonzero ideal $I$ of $R$ such that $(h.\til r)I$ and all $\til r(\ant(h_2).I)$
are contained in $R$; this is possible by continuity of the action $H \acts R$ \cite[Proposition 2.2]{sM93a}.
Using $H$-equivariance of $\rho$ one computes, with $x \in I$,
\[
\begin{aligned}
\til\rho(h.\til r)\rho x  &= \rho((h.\til r) x) = \rho(h_1.(\til r(\ant(h_2).x))) = h_1.\rho(\til r(\ant(h_2).x)) \\
&= h_1.(\til\rho\,\til r\,\rho(\ant(h_2).x)) = h_1.(\til\rho\,\til r\,(\ant(h_2).\rho x)) = (h.\til\rho\,\til r)\rho x,
\end{aligned}
\]
where the last equality uses the identity \eqref{E:HAlgId}.
Thus, $(\til\rho(h.\til r) - h.\til\rho\,\til r) \rho I  = 0$. Since
$(\rho I)\til W$ is a nonzero ideal of $\til W$, it follows that $\til\rho(h.\til r) = h.\til\rho\,\til r$, proving the claim.

The image $\pi T \subseteq W \subseteq \til{W}$
centralizes $\rho R$ and hence also $\til\rho\til R = \rho R\,\til\rho(\cen \til R)$, 
because $\til\rho(\cen \til R) \subseteq \cen \til{W}$.
Therefore, $\til\rho$ and $\pi\big|_{T}$ give a homomorphism
$\til\pi \colon \til V = \til R \otimes T \to \til W$. The preimage of $Q$ in Proposition~\ref{P:bij2} is $\Ker\til\pi$; see 
\cite[proof of Proposition 6]{mL09}. 
Finally, since $\til\rho$ and  $\pi\big|_{T}$ are maps in $\ModAlg H$,
it follows that $\til\pi$ is likewise. Hence $\Ker \til\pi$ is an $H$-ideal, as desired.
\end{proof}


\subsection{Proof of Theorem~\ref{T:Strat}}
\label{SS:ProofThm1}

Let $H$ be a cocommutative Hopf algebra over an algebraically closed field $\k$ and 
let $H \acts A$ be an integral action. Fix a coefficient Hopf algebra $\cO \subseteq H^\circ$
that is an integral domain and write $A_{\cO} = A \otimes \cO$ as before.
Given  $I \in \HSpec A$, our goal is to describe the stratum
$\Spec_IA$. Replacing $A$ by $A/I$, we may assume that
$A$ is prime (Proposition~\ref{P:IntPrime}) and focus on the set $\Spec_0A
=  \{ P \in \Spec A \mid P\byH = 0\}$. 

\subsubsection{}
\label{SSS:Pt1}
Put $X:= \{ Q \in \Spec A_{\cO} \mid Q \cap A = 0\}$.
We first establish a bijection between $\Spec_0A$ and the subset
$X^H \subseteq X$ consisting of all $Q \in X$ that are stable for the $\Dot$-action \eqref{E:dot'}.
Recall that $\Psi$ gives
an automorphism of the algebra $A_{\cO}$ and,
for any $P \in \Spec A$, we have $P':= \Psi(P \otimes \cO) \in \Spec A_{\cO}$;
see Proposition~\ref{P:DotInv} and \eqref{E:IntPrime}. If $P \in \Spec_0A$, then $P' \in X$
by Lemma~\ref{L:Coeff'}(b). So we have a map $\Spec_0 A \to  X$, $P \mapsto   P'$, which
is evidently injective. By Proposition~\ref{P:DotInv}(d), the image consists of $H$-ideals; so $P' \in X^H$.
Conversely, let $Q \in X^H$ be given and put $P = \Phi(Q) \cap A$;
this is a prime ideal of $A$, because $\Phi(Q) \in \Spec A_{\cO}$ and the embedding
$A \into A_{\cO}$ is centralizing. Also, Lemma~\ref{L:Coeff'}(b) gives $P \byH = P' \cap A
\subseteq Q \cap A = 0$; so $P \in \Spec_0 A$. Finally, $Q = P'$ by Proposition~\ref{P:DotInv}(d).
Thus, we have the desired bijection:
\[
\begin{tikzpicture}[baseline=(current  bounding  box.358),  >=latex, scale=.7,
bij/.style={above,sloped,inner sep=0.5pt}]
\matrix (m) [matrix of math nodes, 
column sep=2em, text height=1.5ex, text depth=0.25ex]
{\Spec_0 A & X^H= \{ Q \in \Spec A_{\cO} \mid Q \cap A = 0 \text{ and $Q$ is stable under \eqref{E:dot'}}\} \\ 
\hspace{.2in} P &   P' = \Psi(P \otimes \cO) \\};
\draw[->] (m-1-1) edge node[bij] {$\sim$} (m-1-2);
\draw[|->] (m-2-1) edge (m-2-2);
\end{tikzpicture}
\]
This bijection has the following properties, with (i) being evident and 
(ii) resulting from the algebra isomorphism $A_{\cO}/P' \cong A_{\cO}/(P \otimes \cO)
\cong (A/P)\otimes \cO$ that comes from $\Psi$:
\begin{enumerate}
\item[(i)]
$P_1 \subseteq P_2$ if and only if $P_1' \subseteq  P_2'$, and
\item[(ii)]
$\cC(A_{\cO}/P') \cong \cC((A/P)\otimes \cO)$.
\end{enumerate}


\subsubsection{}
\label{SSS:Pt2}
Now put $K = \cC A$ and $C = K \otimes \cO$; the former is a $\k$-field and the latter a commutative
integral domain, identical to the algebra $C_0$ in \eqref{E:S_I}. 
Let $\til A = AK$ denote the central closure of $A$ 
and put $\til A_{\cO} = \til A \otimes \cO \cong \til A \otimes_K C$. Then we have the following 
isomorphisms of posets (for $\subseteq$),
with $d_1$ coming from Proposition~\ref{P:bij2} and $d_2$ from Proposition~\ref{P:bij1}:
\[
\begin{tikzpicture}
[baseline=(current  bounding  box.center), scale=.8, >=latex,
bij/.style={above,sloped,inner sep=0.5pt}]
\node(1) at (-8.6,0){$ d \colon X = \{ Q \in \Spec A_{\cO} \mid Q \cap A = 0 \}$};
\node(2) at (-0.9,0){$\{ P \in \Spec \til A_{\cO} \mid P \cap \til A = 0 \}$};
\node(3) at (3.9,0){$\Spec C$.};
\draw[->]
(1) edge node[bij] {$\sim$} node[below]{\scriptsize $d_1$} (2)
(2) edge node[bij] {$\sim$} node[below]{\scriptsize $d_2$} (3);
\end{tikzpicture} 
\]
The composite of the bijection ${\phantom i}' \colon \Spec_0 A \iso X^H$ in \S\ref{SSS:Pt1},  
the inclusion $X^H \into X$, and the above bijection $d = d_2 \circ d_1$ yields the map
\[
\begin{tikzpicture}
[baseline=(current  bounding  box.center), scale=.8, >=latex,
bij/.style={above,sloped,inner sep=0.5pt}]
\matrix (m) [matrix of math nodes, row sep=.06em,
column sep=2em, text height=1.5ex, text depth=0.25ex]
{c \colon\Spec_0 A & X^H & X & \Spec C .\\};
\draw[->]
(m-1-1) edge node[bij] {$\sim$} node[below]{\scriptsize ${\phantom i}'$} (m-1-2)
(m-1-3) edge node[bij] {$\sim$} node[below]{\scriptsize $d$} (m-1-4);
\draw[right hook->] (m-1-2) edge (m-1-3);
\end{tikzpicture} 
\]
Since ${\phantom i}'$ and both $d_i$ are order isomorphisms and 
preserve hearts, the same holds for $c$:
\begin{enumerate}
\item[(i)]
$c(P_1) \subseteq c(P_2)$ if and only if $P_1 \subseteq P_2$, and
\item[(ii)]
$\Fract(C/c(P)) = \cC(C/c(P)) \cong \cC((A/P) \otimes \cO)$.
\end{enumerate}
Finally, $\Im c$ consists exactly of the $H$-ideals in $\Spec C$. This amounts to 
checking that  $Q \in X$ is an $H$-ideal if and only if $d(Q) \in \Spec C$ is an $H$-ideal.
But, by Lemma~\ref{L:C(R)}(a) and Proposition~ \ref{P:DotInv}(c),
$\til A_{\cO} \in \ModAlg H$ for the $\Dot$-action, 
with $A_{\cO}$ and $C = \cen(\til A_{\cO})$ being $H$-module subalgebras.
By Propositions~\ref{P:bij1} and \ref{P:bij2}, we know that a given 
$P \in \Spec \til A_{\cO}$ with $P \cap \til A = 0$ is an $H$-ideal if and only if 
$d_2(P) = P \cap C \in \Spec C$ is an $H$-ideal and if and only if 
$d_1^{-1}(P) = P \cap A_{\cO} \in X$ is an $H$-ideal.
This completes the proof of Theorem~\ref{T:Strat}. \qed


\subsection{Tensor algebras}
\label{SS:Tens}

As an application of Theorem~\ref{T:Strat}, we offer the following result on the tensor algebra $\Tens V$
of a finite-dimensional representation $V \in \Rep H$. The $H$-action on $V$ extends uniquely to an
action $H \acts \Tens V$ \cite[10.4.2]{mL18}.

\begin{cor}
\label{C:Tens}
Let $H$ be a cocommutative Hopf algebra over an algebraically closed field $\k$ and 
let $V$ be a representation of $H$ with $2 \le \dim_\k V < \infty $ and such
that the coefficient Hopf algebra of $V$ (\S\ref{SSS:Coeff2}) is a domain.
Then every nonzero prime ideal of the tensor algebra $\Tens V$ contains a nonzero $H$-stable prime ideal.
\end{cor}

\begin{proof}
The action $H \acts A = \Tens V$ is integral and we may apply 
Theorem~\ref{T:Strat} with $I = 0$. By a result of Kharchenko \cite{vK78}, $\Q A = A$
and so $\cC A = \cen A = \k$.  Therefore, the algebra $C_0$ in  Theorem~\ref{T:Strat} coincides
with $\cO$, the coefficient Hopf algebra of the representation $V$. Since $C_0 = \cO$ is
$H$-simple by Lemma~\ref{L:Stability}, we have $\Spec^H\!C_0 = \{ 0\}$ and Theorem~\ref{T:Strat}
gives $\Spec_0 A = \{0\}$. Consequently, if $P$ is a nonzero prime ideal of $A$, then $P \notin \Spec_0 A$
and so $P\byH \neq 0$. Finally, $P\byH$ is an $H$-stable prime ideal by Proposition~\ref{P:IntPrime}.
\end{proof}


\subsection{Torus actions}
\label{SS:Torus}

For rational torus actions, the set $\Spec^H\!C_I $ in Theorem~\ref{T:Strat} has a simpler 
description as $\Spec Z_I$ for an explicit commutative domain $Z_I$. This was shown in \cite{mL13},
but the presentation contains some inaccuracies which we will now repair.

To start with, let $H$ be pointed cocommutative and consider an arbitrary locally finite action $H \acts A$
with coefficient Hopf subalgebra 
$\cO \subseteq H^\circ$. Associated to any $I \in \HSpec A$,  
we have the commutative algebra $C_I = \cC(A/I) \otimes \cO \in \ModAlg H$
as in \eqref{E:S_I}; the $H$-action on $C_I$ is \eqref{E:dot'} using the action $H \acts \cC(A/I)$ from
Lemma~\ref{L:C(R)}(a). This action need not be locally finite (\S\ref{SSS:QA}). 
Following \cite{mL13}, we consider the locally finite part and the invariants:
\begin{equation*}
Z_I:= \fin{\cC(A/I)} \supseteq F_I:= \cC(A/I)^H .
\end{equation*}
The invariant algebra $F_I$ is in fact a field \cite[Lemma 1.4]{jM91}.
By Lemma~\ref{L:C(R)}(b), the action $H \acts Z_I$ has coefficient Hopf algebra $\cO$
and Proposition~\ref{P:DotInv}(c) gives the following isomorphism which
generalizes \cite[Equation (2)]{mL13}:
\begin{equation}
\label{E:Z_I2}
Z_I \cong \Psi(Z_I) = (C_I)^H.
\end{equation}

Turning to torus actions, let $\k$ be algebraically closed and let
$G$ be an algebraic $\k$-torus acting rationally on the $\k$-algebra $A$. 
As was remarked in \S\ref{SS:StratExamples}, the algebra of polynomial functions $\cO = \cO(G)$ serves as a
coefficient Hopf algebra for this action. Moreover, $\cO$ is a Laurent polynomial algebra over $\k$, the algebra
$C_I$ is a Laurent polynomial algebra over the field $\cC(A/I)$, and
the isomorphism $Z_I \cong (C_I)^G$ in \eqref{E:Z_I2} realizes $Z_I$ as a
Laurent polynomial algebra over the field of $G$-invariants $F_I = \cC(A/I)^G$; see 
the Stratification Theorem in \cite{mL13} or \cite[Equation (4)]{mL13}.
The set $\Spec^H\!C_I$ in Theorem~\ref{T:Strat} now is the set $\Spec^G\!C_I$ consisting of all
$G$-stable prime ideals of $C_I$ and $\Spec\, (C_I)^G \cong  \Spec Z_I$ by \eqref{E:Z_I2}.

\begin{lem}[notation as above]
\label{L:Torus}
Contraction and extension of ideals give a $G$-equivariant order isomorphism
$\Spec^G\!C_I \iso \Spec\, (C_I)^G$.
\end{lem}

\begin{proof}
By \cite[Equation (5)]{mL13}, all $G$-stable ideals of $S:=C_I$ are generated
by their intersection with $R:= S^G$.  Therefore, the contraction 
map $\Spec^G\!S \to \Spec R$, $\fP \mapsto \fP \cap R$, is injective.  
If $\fa$ is an ideal of $R$, then its extension $\til\fa:= \fa S$ is a $G$-stable ideal of $S$.
Moreover, $\til\fa\cap R = \fa$, because $S$ is free over $R$ by \cite[2.3]{mL13}. Thus, $\til\fa$ is the
only $G$-stable ideal of $S$ that contracts to $\fa$. Now let $\fp \in \Spec R$ be given and
choose an ideal $\fP$ of $S$ maximal subject to the condition $\fP \cap R = \fp$. Then $\fP$ is prime
and hence so is its $G$-core by Proposition~\ref{P:IntPrime}. 
Thus, $\fP\byG \in \Spec^G\!S$ and $\fP\byG = \til\fp$ by the foregoing. 
This proves surjectivity of the contraction map and that the inverse is given by extension.
The statements about $G$-equivariance and preservation of inclusions in the lemma are clear.
\end{proof}


\section{Semiprimeness}
\label{S:Semi}


\subsection{Reformulations}
\label{SS:Reformulation}

We start this section by giving several reformulations, in terms of the semiprime radical
operator $\sqrt{\bdot\phantom{!}}$, of the conclusion of Theorem~\ref{T:Semi},
which is equivalent to (ii) in the lemma below. 
The semiprime radical of a subset $X \subseteq A$, by definition,
is the unique smallest semiprime ideal of $A$ containing $X$:
\[
\sqrt X = \bigcap_{\substack{ P \in \Spec A \\ P \supseteq X}} P\,.
\] 
We continue to assume that $A \in \ModAlg H$; the Hopf algebra $H$ can be arbitrary for now.

\begin{lem}
\label{L:Reformulation}
The following are equivalent:
\begin{enumerate}
\item[(i)]
If $J$ is an $H$-ideal of $A$, then so is $\sqrt J$;
\item[(ii)]
for all ideals $I$ of $A$, the $H$-core $\sqrt{I}\byH$ is semiprime;
\item[(iii)]
$H.\sqrt I \subseteq \sqrt{H.I}$ for any ideal $I$ of $A$.
\end{enumerate}
\end{lem}

\begin{proof}
(i) $\Rightarrow$ (ii). 
We may assume that $I$ is semiprime. Then $\sqrt{I\byH} \subseteq \sqrt I =I$, since 
$\sqrt{\bdot\phantom{!}}$ preserves inclusions. In fact, $\sqrt{I\byH} \subseteq I\byH$,
because $\sqrt{I\byH}$ is an $H$-ideal by (i).
The reverse inclusion being trivial, it follows that $I\byH = \sqrt{I\byH}$ is semiprime.

(ii) $\Rightarrow$ (iii). 
Let $J$ denote the ideal of $A$ that is generated by the subset $H.I \subseteq A$. Then 
$\sqrt I \subseteq \sqrt J = \sqrt{H.I}$ and $J$
is easily seen to be an $H$-ideal. (If the antipode $\ant$ is bijective, then $J = H.I$ \cite[Exercise 10.4.3]{mL18}.)
Thus, $J = J\byH \subseteq \sqrt{J}\byH$ and the latter ideal
is semiprime by (ii). It follows that $\sqrt J = \sqrt{J\byH} \subseteq \sqrt{J}\byH$.
Again, the reverse inclusion is clear; so $\sqrt{J} = \sqrt{J}\byH$.
Therefore, $H.\sqrt I \subseteq H.\sqrt J = \sqrt J = \sqrt{H.I}$.

(iii) $\Rightarrow$ (i). Specialize (iii) to the case where $I=J$ is an $H$-ideal. 
\end{proof}


\subsection{Extending the base field} 
\label{SS:Reduction}

For the proof of Theorem~\ref{T:Semi}, we may work over an algebraically closed base field. 
This follows by taking $K$ to be an algebraic closure of $\k$ in the argument below.

Let $K/\k$ be any field extension and put $H' = H \otimes K$
and $A' = A \otimes K$. Then $A' \in \ModAlg{H'}$ and
$H'$ is cocommutative if $H$ is so. Assuming Theorem~\ref{T:Semi} to hold for $A'$, our goal is to show
that it also holds for $A$. So let $I$ be a semiprime ideal of $A$. 
Viewing $A$ as being contained in $A'$ in the usual way, $IA'$ is an ideal of $A'$ satisfying
$IA' \cap A = I$. By Zorn's Lemma, we may choose an ideal $I'$ of $A'$ that is
maximal subject to the condition $I'\cap A = I$. Then $I'$ is semiprime. For, if $J$ is any ideal of $A'$ 
such that $J \supsetneqq I'$, then $J\cap A \supsetneqq I$ by maximality of $I'$, and so $(J\cap A)^2
\nsubseteq I$ by semiprimeness of $I$. Since $(J\cap A)^2 \subseteq J^2 \cap A$, it follows that $J^2 
\nsubseteq I'$, proving that $I'$ is semiprime. Therefore, by our assumption, the core $I'\text{\sl :}H'$ is
semiprime. Since the extension $A \into A'$ is centralizing, it follows that $(I'\text{\sl :}H') \cap A$
is a semiprime ideal of $A$. Finally,  
$(I'\text{\sl :}H') \cap A = \{a \in A\mid H'.a \subseteq I'\} = \{a \in A \mid H.a \subseteq I'\cap A = I \} = I\byH$
by \eqref{E:IbyH}, giving the desired conclusion that $I\byH$ is semiprime. 


\subsection{Enveloping algebras}
\label{SS:Env}

\subsubsection{} 
\label{SSS:Power}
For any ring $R$, let $R\llbracket X_\lambda \rrbracket_{\lambda\in \Lambda}$ 
denote the ring of formal power series in the 
commuting variables $X_\lambda$ $(\lambda \in \Lambda)$
over $R$; see \cite[Chap.~III, \S 2, n$^{\rm o}$ 11]{nB70}.

\begin{lem} 
\label{L:primeness}
Let $R$ be a ring,
let $\Lambda$ be any set, and let $S$
be a subring of $R\llbracket X_\lambda \rrbracket_{\lambda\in \Lambda}$
such that $S$ maps onto $R$ under the homomorphism
$R\llbracket X_\lambda \rrbracket_{\lambda\in \Lambda} \to R$, 
$X_\lambda \mapsto 0$. If $R$ is prime (resp., semiprime, a domain) then so is $S$.
\end{lem}

\begin{proof}
We write monomials in the variables $X_\lambda$ as 
$X^{\bn} = \prod_\lambda X_\lambda^{\bn(\lambda)}$ $(\bn \in M)$, where
$M = \ZZ_+^{(\Lambda)}$ denotes the additive monoid of all functions 
$\bn \colon \Lambda \to \ZZ_+$ such that $\bn(\lambda) = 0$ for almost all $\lambda \in \Lambda$. 
Fix a total order $<$ on $M$ 
having the following properties (e.g., \cite[Example 2.5]{mAcH07}): 
every nonempty subset of $M$ has a smallest element; the zero function $\bm{0}$
is the smallest element of $M$; and 
$\bn < \bm{m}$ implies $\bn + \bm{r} < \bm{m}+ \bm{r}$ for all $\bn, \bm{m}, \bm{r} \in M$.

For any $0 \neq s  = \sum_{\bn \in M} s_{\bn}X^{\bn} \in 
R\llbracket X_\lambda \rrbracket_{\lambda\in \Lambda}$, 
we may consider its lowest coefficient, $s_{\min}:= s_{\bm{m}}$ with 
$\bm{m} = \min \{ \bn \in M \mid s_{\bn} \neq 0 \}$.
If $R$ is prime and $0 \neq s,t \in S$ are given, then 
$0 \neq s_{\min}rt_{\min}$ for some $r\in R$. 
By assumption, there exists an element $u \in S$ having the form
$u = r + \sum_{\bn \neq \bm{0}} u_{\bn}X^{\bn}$. It follows that $sut \neq 0$, with
$(sut)_{\min} = s_{\min}rt_{\min}$. This proves that $S$
is prime. For the assertions where $R$ is semiprime or a domain, take $s=t$ or $r=1$,
respectively.
\end{proof}

\subsubsection{}
\label{SSS:Env2}
Now let $H = U\fg$ be the enveloping algebra of an arbitrary Lie $\k$-algebra $\fg$ and assume that $\ch \k = 0$.
For the proof of the next proposition, we recall the structure of the convolution algebra $\Hom_\k(H,R)$ for 
an arbitrary $\k$-algebra $R$. Let $(e_\lambda)_{\lambda\in \Lambda}$ be a 
$\k$-basis of $\fg$ and fix a total order of the index set $\Lambda$. Put
$M = \ZZ_+^{(\Lambda)}$ as in the proof of Lemma~\ref{L:primeness} and,
for each $\bn \in M$, put 
$e_{\bn} = \prod_{\lambda}^{<} \frac{1}{\bn(\lambda)!}\,e_\lambda^{\bn(\lambda)} \in H$, where the superscript $<$
indicates that the factors occur in the order of increasing $\lambda$. The elements $e_{\bn}$ form
a $\k$-basis of $H$ by the Poincar\'e-Birkhoff-Witt Theorem, and the comultiplication of
$H$ is given by 
$\Delta e_{\bn} = \sum_{\br + \bs = \bn} e_{\br} \otimes e_{\bs}$; see \cite[Example 9.5]{mL18}.
Writing $X^{\bn} = \prod_\lambda X_\lambda^{\bn(\lambda)}$ as in the proof of Lemma~\ref{L:primeness}, we obtain 
an isomorphism of $\k$-algebras, 
\[
\phi \colon \Hom_\k(H,R) \iso R\llbracket X_\lambda \rrbracket_{\lambda\in \Lambda}\,, 
\quad f \mapsto \sum_{\bn \in M} f(e_{\bn})X^{\bn} .
\]
Under this isomorphism, the algebra map
$u^* \colon \Hom_\k(H,R) \to R$\,, $f \mapsto f(1)$\,, coming from the
unit map $u = u_H \colon \k \to H$ translates into the map
$R\llbracket X_\lambda \rrbracket_{\lambda\in \Lambda} \onto R$, 
$X_\lambda \mapsto 0$, as considered in Lemma~\ref{L:primeness}.

The proposition below is not new; it can be found in Dixmier's book
\cite[3.3.2 and 3.8.8]{jD96}, albeit with a very different proof. 
Recall that an ideal of a ring is said to be \emph{completely prime} if the quotient
is a domain.

\begin{prop} 
\label{P:primeness}
Let $H = U\fg$ be the enveloping algebra of a Lie $\k$-algebra $\fg$,
let $A \in \ModAlg H$, and let $I$ be an ideal of $A$. Assume that $\ch\k = 0$. If
$I$ is prime, semiprime or completely prime, then $I\byH$ is 
likewise. 
\end{prop}

\begin{proof}
By \eqref{E:IbyH}, the core $I\byH$ is identical to the kernel of the map
$\d_I \colon A \to \Hom_\k(H,A/I)$ that is given by $\d_I(a) = (h \mapsto h.a + I)$.
We need to show that the properties of being prime, semiprime, or a domain
all transfer from $A/I$ to the subring $\d_IA \subseteq \Hom(H,A/I)$ or, equivalently, to the subring
$S \subseteq (A/I)\llbracket X_\lambda \rrbracket_{\lambda\in \Lambda}$ that corresponds to $\d_IA$
under the above isomorphism $\phi$. Consider the map $u^* \colon \Hom_\k(H,A/I) \to A/I$\,,
$f \mapsto f(1)$, and note that $(u^*\circ\d_I)(a) = a+I$ for $a \in A$. 
Therefore, $S$ maps onto $A/I$
under the map $(A/I)\llbracket X_\lambda \rrbracket_{\lambda\in \Lambda} \to A/I$, 
$X_\lambda \mapsto 0$. Now all assertions follow from Lemma~\ref{L:primeness}. 
\end{proof}

\subsubsection{}
\label{SSS:Env3}
For an arbitrary cocommutative Hopf algebra $H$, we cannot expect a result as strong as 
Proposition~\ref{P:primeness}: group algebras provide easy counterexamples
to the primeness and complete primeness assertions.  
Indeed, let $H = \k G$ be the group $\k$-algebra of the group $G$ and let $A \in \ModAlg H$.
Then $I\byH = \bigcap_{g \in G} g.I$ for any ideal $I$ of $A$. If $I$ is
semiprime, then so are all $g.I$, because each $g \in G$ acts on $A$ by algebra automorphisms, and
hence $\bigcap_{g \in G} g.I$ will be semiprime also. However, primeness and complete primeness, while
inherited by each $g.I$, are generally
lost upon taking the intersection.


\subsection{Proof of Theorem~\ref{T:Semi}}
\label{SS:ProofSemi}

Let $H$ be cocommutative Hopf and assume that $\ch\k = 0$ and that $\k$ is
algebraically closed, as we may by \S \ref{SS:Reduction}.
Then $H$ has the structure of a smash product,  
$H \cong U\# V$, where $U$ is the enveloping algebra of the Lie algebra
of primitive elements of $H$ and $V$ is the group algebra of the group of grouplike 
elements of $H$; see \cite[\S 13.1]{mS69} or \cite[\S 15.3]{dR12}. 
Thus, both $U$ and $V$ are Hopf subalgebras of $H$ and
$H = UV$, the $\k$-space spanned by all products $uv$ with $u \in U$ and $v \in V$.
Viewing $A \in \ModAlg U$ and $A \in \ModAlg V$ by restriction,
repeated application of \eqref{E:IbyH} gives the following equality for any ideal $I$ of $A$:
\[
I\byH = \{ a \in A \mid UV.a \subseteq I\} = \{ a \in A \mid V.a \subseteq I\text{\sl :} U\} = (I\text{\sl :}U)\text{\sl :} V.
\]
If $I$ is semiprime, then so is $I\text{\sl :}U$ (Proposition~\ref{P:primeness}). Our remarks
on group algebras in the first paragraph of this proof further give semiprimeness of $(I\text{\sl :}U)\text{\sl :} V$.
Thus, $I\byH$ is semiprime and Theorem~\ref{T:Semi} is proved. \qed


\begin{ack}
We thank the referee for pointing out reference \cite{jD96} in 
connection with Proposition~\ref{P:primeness} above.
\end{ack}


\def\cprime{$'$}
\providecommand{\bysame}{\leavevmode\hbox to3em{\hrulefill}\thinspace}
\providecommand{\MR}{\relax\ifhmode\unskip\space\fi MR }
\providecommand{\MRhref}[2]{%
  \href{http://www.ams.org/mathscinet-getitem?mr=#1}{#2}
}
\providecommand{\href}[2]{#2}


\end{document}